\documentclass[10pt]{amsart}
\usepackage{amsmath,amssymb,verbatim}
\usepackage[latin1]{inputenc}

\usepackage{enumerate}
\usepackage{pdflscape}
\usepackage{amsthm}

\usepackage{mathrsfs}
\usepackage[dvips]{graphicx}
\usepackage{color}
\usepackage[normalem]{ulem}
\usepackage{tikz}
\usepackage[all]{xy}

\newtheorem{theorem}{Theorem}[section]
\newtheorem{lemma}[theorem]{Lemma}
\newtheorem{proposition}[theorem]{Proposition}

\newtheorem{remark}[theorem]{Remark}

\numberwithin{equation}{section}

\newcommand{\pa}[2]{\varepsilon_{#2}(#1)}
\newcommand{\qand}{\hspace{0.5cm} \text{and} \hspace{0.5cm}}

\newcommand{\Aut}{\mbox{\rm Aut}}

\newcommand{\Z}{{\mathbb Z}}
\newcommand{\Q}{{\mathbb Q}}

\newcommand{\C}{{\mathbb C}}

\newcommand{\K}{{\mathbb K}}
\newcommand{\LL}{{\mathbb L}}

\newcommand{\tr}{{\rm tr}}
\newcommand{\Soc}{{\rm Soc}}

\newcommand{\matriz}[1]{\begin{array} #1 \end{array}}

\newcommand{\GEN}[1]{\left\langle #1 \right\rangle}

\newcommand{\diag}{\operatorname{diag}}
\newcommand{\Cl}{\operatorname{Cl}}

\newcommand{\Exp}{\operatorname{Exp}}

\newcommand{\x}{\delta}
\newcommand{\y}{\gamma}

\newenvironment{proofof}{\par\bigskip\noindent{\it Proof of }}{\qed\par\bigskip}

\title[Zassenhaus Conjecture for cyclic-by-nilpotent groups]{On the Zassenhaus Conjecture for certain cyclic-by-nilpotent groups}

\author{Mauricio Caicedo}
\address{Department of Mathematics, Vrije Universiteit Brussel,  Pleinlaan 2, 1050 Brussels, Belgium}
\email{mcaicedo@vub.ac.be}

\author{\'Angel del R\'io}
\address{Departamento de Matem\'aticas, Universidad de Murcia,  30100 Murcia, Spain}
\email{adelrio@um.es}

\thanks{
This research is partially supported by the Research Foundation Flanders (FWO - Vlaanderen),
the Spanish Government under Grant MTM2016-77445-P with ``Fondos FEDER'' and, by Fundaci\'on S\'eneca of Murcia under Grant 19880/GERM/15}

\begin{document}

\begin{abstract}
Hans Zassenhaus conjectured that every torsion unit of the integral group ring of a finite group $G$ is conjugate within the rational group algebra to an element of the form $\pm g$ with $g\in G$. 
This conjecture has been disproved recently for metabelian groups, by Eisele and Margolis. However it is known to be true  for many classes of solvable groups, as for example nilpotent groups, cyclic-by-abelian groups and groups having a Sylow subgroup with abelian complement.  
On the other hand, the conjecture remains open for the class of supersolvable groups. 
This paper is a contribution to this question. More precisely, we study the conjecture for the class of cyclic-by-nilpotent groups with special attention to the class of cyclic-by-Hamiltonian groups. 
We prove the conjecture for cyclic-by-$p$-groups and  some cyclic-by-Hamiltonian groups.
\end{abstract}

\maketitle

\section{Introduction}

In this paper $G$ is a finite group and $\Z G$ denotes the group ring of $G$ with coefficients in $\Z$.
The study of the group of units of $\Z G$ has been an active field of research since the seminal work of Graham Higman \cite{Higman1940Thesis}.
Let $V(\Z G)$ denote the group formed by the units of $\Z G$ of augmentation $1$. 
Then the group of units of $\Z G$ is $\pm V(\Z G)$ and therefore we can restrict our attention to the group $V(\Z G)$.
One of the main aims in this area, which is partially motivated by the Isomorphism Problem,  is the description of the torsion units of $\Z G$.
In the 1970's, Hans Zassenhaus conjectured that every torsion element $u$ of $V(\Z G)$ is conjugate  within the rational group algebra $\Q G$ to an element $g$ of $G$ \cite{Zassenhaus1974}.
In such case one says that $u$ and $g$ are rationally conjugate. 
Eisele and Margolis have shown recently a counterexample to the Zassenhaus Conjecture \cite{EiseleMargolis2017}. Although the Zassenhaus Conjecture is not the ``correct solution'' to the original problem of describing the torsion units of $\Z G$, yet it provides an answer for many solvable groups and some non-solvable groups (see  \cite{MargolisdelRioSurvey} for a recent survey). 
It would be nice to have a classification of the groups for which the Zassenhaus Conjecture holds. This looks out of reach at this moment, but proving or disproving the Zassenhaus Conjecture for large classes of groups is a way to converge to  this classification. The counterexample of Eisele and Margolis shows that the Zassenhaus Conjecture fails for metabelian groups. However the conjecture holds, among others, for nilpotent groups \cite{Weiss1991}, cyclic-by-abelian groups \cite{CaicedoMargolisdelRio2013} and for groups having a Sylow subgroup with abelian complement \cite{Hertweck2006}. This information, as well as some ideas in the proofs of these results suggest that the Zassenhaus Conjecture might hold for supersolvable groups or at least for cyclic-by-nilpotent groups.
The aim of this paper is to make some contributions in that direction. 
Actually, our original motivation was to study the Zassenhaus Conjecture for cyclic-by-Hamiltonian groups, a question posed by Kimmerle to us in a private conversation.

Another aim of this paper is to consider a special case of the Zassenhaus Conjecture which appeared as Research Problem 35 in \cite{Sehgal1993}.
If $N$ is a normal subgroup of $G$ then $V(\Z G,N)$ denotes the group of units of $\Z G$ mapped to $1$ by the natural homomorphism $\Z G\rightarrow \Z(G/N)$. 
Sehgal's Problem asks whether every torsion  element of $V(\Z G,N)$ is rationally conjugate to an element of $G$ provided $N$ is nilpotent. 
Actually, the torsion unit in the counterexample of Eisele and Margolis is a negative solution to Sehgal's Problem for some group $G$ and a normal subgroup $N$ of $G$ with $N$ and $G/N$ abelian.
This raises the question of when Sehgal's Problem has a positive solution.

Our main results consists in proving the Zassenhaus Conjecture for two subclasses of the class of cyclic-by-nilpotent groups and to give a positive solution for Sehgal's Problem for a special case. More precisely we prove the following results:

\begin{theorem}\label{G'OneNonCyclic}
	Let $N$ be a nilpotent subgroup of $G$ containing $G'$ and suppose that $N_{p'}$ is abelian and $G'_{p'}$ is cyclic for some prime $p$. Then every torsion element of $V(\Z G,N)$ is rationally conjugate to an element of $N$.
\end{theorem}

\begin{theorem}\label{CyclicBypGroup}
	If $G$ is a cyclic-by-$p$-group (i.e. $G$ has a cyclic normal subgroup of prime power index) then every torsion element of $V(\Z G)$  is rationally conjugate  to an element of $G$.
\end{theorem}

\begin{theorem}\label{CyclicByHamiltonian}
Let $G$ be a finite group having a normal cyclic subgroup $A$ such that $G/A$ is Hamiltonian and one of the following conditions holds:
\begin{enumerate}
 \item $|A|$ is not multiple of $8$;
 \item $[G:A]$ is not multiple of $16$;
 \item the image of the action of $G$ on the Sylow $2$-subgroup of $A$ has order different from $4$.
\end{enumerate}
Then every torsion element of $V(\Z G)$ is rationally conjugate to an element of $G$.

\end{theorem}

By Theorem~\ref{G'OneNonCyclic}, in all the negative solutions to Sehgal's Problem, with $N$ and $G/N$ abelian, $G'$ has at least two non-cyclic Sylow subgroups. 
This seems to be quite sharp because, by a result of Cliff and Weiss \cite{CliffWeiss}, if $N$ has at most one non-cyclic Sylow subgroup then Sehgal's Problem has a positive solution for every $G$, while for some of the counterexample in \cite{EiseleMargolis2017} one has $G'=N\cong C_p^2\times C_q^2$ with $p$ and $q$ different primes.

Theorem~\ref{G'OneNonCyclic} will be used in the proof of Theorem~\ref{CyclicBypGroup} and in the proof of the following proposition which gives information on the $2$-Sylow subgroup of a hypothetical minimal cyclic-by-Hamiltonian counterexample to the Zassenhaus Conjecture. The proposition is the bulk of the proof of Theorem~\ref{CyclicByHamiltonian}.

\begin{proposition}\label{MinimalCounterexampleHamiltonian}
Let $G$ be a finite group with a normal cyclic subgroup $A$ of order $n$ such that $G/A$ is Hamiltonian. Suppose that the Zassenhaus Conjecture holds for every proper quotient of $G$ but fails for $G$. Then $8|n$ and there are $x\in C_G(A)$ and $y,w\in G$ such that $x^2\not\in A$ and if $v=(y,w)$ then $v$ generates the Sylow 2-subgroup of $A$ and the following conditions holds:
$$|x|=4,\; w^2=x^2y^2=1, \;v^w = v^{-1}, \;
v^y = v^{\frac{n}{2}-1}, \;x^w = v^{\frac{n}{4}}x, \;
x^y \in x^{-1}\GEN{v^{\frac{n}{2}}}.$$ 
\end{proposition}

The paper is structured as follows. 
In Section~\ref{SectionPreliminaries} we establish the basic notation and collect some known facts about the Zassenhaus Conjecture which will be used in the subsequent sections. For example, we recall the role of partial augmentations which gives the necessary background to prove Theorem~\ref{G'OneNonCyclic} by combining results by Hertweck, Margolis and del Río. 
Section~\ref{SectionCycByp} starts proving general results for cyclic-by-nilpotent groups and finishes with the proof of Theorem~\ref{CyclicBypGroup}. 
Section~\ref{SectionFormulaPA} revisits an equality from \cite{CaicedoMargolisdelRio2013} concerning partial augmentations, which plays an important role in the proofs of Proposition~\ref{MinimalCounterexampleHamiltonian} and Theorem~\ref{CyclicByHamiltonian}, which are given at the end in Section~\ref{SectionCyclicByHamiltonian}.

\section{Notation and Preliminaries}\label{SectionPreliminaries}

The cardinality of a set $X$ is denoted by $|X|$ and $\varphi$ denotes the Euler's totient function. For every integer $n$, we let $\zeta_n$ denote a fixed complex
primitive root of unity of order $n$ and $C_n$ denotes an arbitrary cyclic group of order $n$. If $F/K$ is an extension of number fields then $\tr_{F/K}$ denotes the trace map of $F$ over $K$.

All throughout $G$ is a finite group, $\Cl(G)$ denotes the set of conjugacy classes of $G$, $Z(G)$ denotes the center of $G$, $G'$ stands for the commutator subgroup of $G$, $\Exp(G)$ for the exponent of $G$ and $\Soc(G)$ for the socle of $G$, i.e. the subgroup generated by the minimal (non-trivial) normal subgroups of $G$.
If $g,h \in G$, then $|g|$ denotes the order of $g$, $g^h=h^{-1}gh$, $(g,h)=g^{-1}h^{-1}gh$ and $g^G$ denotes the conjugacy class of $g$ in $G$.
If $X,Y \subseteq G$, then $\GEN{X}$ denotes the subgroup generated by $X$,
$C_X(Y)$ denotes the centralizer of $Y$ in $X$, $N_X(Y)$ denotes the normalizer of $Y$ in $X$ and we denote $(X,Y)=\{(x,y) : x\in X, y\in Y\}$.

If $g$ is an element of order $n$ in a group then the proper powers of $g$ are the elements of the form $g^d$ with $d\mid n$ and $d\ne 1$.
If $\pi$ is a set of prime integers then $g_{\pi}$ and $g_{\pi'}$ denote the $\pi$ and $\pi'$ parts of $g$, respectively.
The notation $G_{\pi}$ (respectively, $G_{\pi'}$) refers to a Hall $\pi$-subgroup (respectively, Hall $\pi'$-subgroup) of $G$.
Furthermore, if $p$ is a prime integer then $g_p$, $g_{p'}$, $G_p$ and $G_{p'}$ are abbreviations of $g_{\{p\}}$, $g_{\{p\}'}$, $G_{\{p\}}$ and $G_{\{p\}'}$, respectively.
All the finite groups $G$ appearing in this paper are solvable, so the existence of $G_{\pi}$ and $G_{\pi'}$ is warranted and these subgroups are unique up to conjugation in $G$ (see e.g. \cite[9.1.7]{Robinson1982}).
Moreover, if $G$ is nilpotent then $G_{\pi}$ and $G_{\pi'}$ are unique.

If $R$ is a ring and $N$ is a normal subgroup of $G$ then the natural map $G\rightarrow G/N$ extends to a ring homomorphism $\omega_N:RG \rightarrow R(G/N)$. Furthermore,  we denote by $V(RG,N)$ the group of units of $RG$ mapped to the identity via $\omega_N$.

If $a=\sum_{x\in G} a_xx\in RG$, with $a_x\in R$ for every $x\in G$, and $X$ is a subset of $G$ then we denote
	$$\varepsilon_X(a)=\sum_{x\in X} a_x.$$
If $g\in G$ then the \emph{partial augmentation} of $a$ at $g$ is $\pa{a}{g^G}$.
For brevity, when $G$ is our target group, the partial augmentation of $a$ at $g$ is simply denoted $\pa{a}{g}$.
The following formula, for $N$ a normal subgroup of $G$, is easy to check

	\begin{equation}\label{PAQuotient}
	\pa{\omega_N(a)}{\omega_N(g)^{G/N}} = \pa{a}{g^GN} = \sum_{X\in \Cl(G), X\subseteq g^GN} \pa{a}{X},
	\end{equation}
where $g^GN=\{xn : x\in g^G, n\in N\}$.

Our main tool is the following well known result.

\begin{proposition}\label{ZC1PAugmentation}\cite{MarciniakRitterSehgalWeiss1987}
Let $G$ be finite group and let  $u$ be a torsion element of $V(\Z G)$.
Then $u$ is rationally conjugate to an element of $G$ if and only if $\varepsilon_{g^G}(u^d)\ge 0$ for every $d\mid n$ and every $g\in G$.
\end{proposition}

We can now give the

\begin{proofof} {\it Theorem~\ref{G'OneNonCyclic}}.
	Let $u$ be a torsion element of $V(\Z G,N)$.
	By the Berman-Higman Theorem there is $b\in N$ with $\omega_{G'}(u)=\omega_{G'}(b)$.
	By the main result of \cite{MargolisHertweck2017} there is $c\in N$ such that for every prime $q$ we have that $u_q$ is conjugate in $\Z_qG$ to $c_q$, where $\Z_q$ denotes the ring of $q$-adic integers. 
	By \cite[Lemma~2.2]{Hertweck2008}, if $\varepsilon_x(u)\ne 0$ then $x_q$ is conjugate to $c_q$ in $G$ and, in particular, $x\in N$. 
	Thus, if $\varepsilon_x(u)\ne 0$ then $\omega_{G'}(x_q)=\omega_{G'}(c_q) = \omega_{G'}(u_q)=\omega_{G'}(b_q)$ for every prime integer $q$, and hence $xG'=cG'=bG'$.
	This proves that if $\varepsilon_x(u)\ne 0$ then $x\in bG'$.

	Let $K$ be a subgroup of $N_{p'}$ which is maximal with respect to the following condition: $K\cap G'_{p'}=1$. 
	Then $N/K$ is cyclic. 
	By \cite[Proposition~1.1]{MargolisdelRioCW}, $0\le \sum_{k\in K} |C_G(xk)|\varepsilon_{xk}(u)$ for every $x\in N$, 
	however, by the previous paragraph, if $k\in K$ and $\pa{u}{xk}\neq 0$ then $xk\in bG'=xG'$. Hence $k\in K\cap G'\cap N_{p'}=K\cap G'_{p'}=1$.
	Therefore $0\le \sum_{k\in K}|C_G(xk)|\pa{u}{xk}=|C_G(x)|\pa{u}{x}$. 
	This proves that $\pa{u}{x}\ge 0$ for every $x\in G$. 
	Moreover, by Proposition 3.1 and Theorem~3.3 of \cite{MargolisdelRioPAP}, we have $\varepsilon_g(u^d)\geq 0$ for every $g\in G$ and every $d\mid n$. Then $u$ is rationally conjugate to an element $g$ of $G$, by Proposition~\ref{ZC1PAugmentation}. Since $u\in V(\Z G,N)$, necessarily $g\in N$. 
\end{proofof}

Another important tool is the following result \cite[Remark 2.4]{Hertweck2008}. 

\begin{lemma}\label{Hertweckaumentocero}
 Let $u$ be a torsion element of $V(\Z G)$ and let $g\in G$. 
 If $\pa{u}{g}\ne 0$ then the order of $g$ divides the order of $u$.
\end{lemma}

From now on we use (ZC) as an abbreviation of ``the Zassenhaus Conjecture".

\begin{remark}\label{Formulaaumentoscociente}
In several arguments we will assume that (ZC) holds for proper quotients of $G$ or for proper powers of a particular torsion element $u$ of $V(\Z G)$.

\begin{enumerate}
\item\label{PAQuotientHyp} In the first case, by \eqref{PAQuotient}, we have
	$$\pa{u}{g^GN}\ge 0 \hspace{0.5cm} \text{for every } g\in G \text{ and every } 1\ne N \unlhd G.$$

\item In the second case, by Proposition~\ref{ZC1PAugmentation}, to prove that $u$ is rationally conjugate to an element of $G$ it suffices to show that $\pa{u}{g}\ge 0$ for every $g\in G$.
\end{enumerate}

\end{remark}

\section{The Zassenhaus Conjecture holds for cyclic-by-p-groups}\label{SectionCycByp}

In this section we prove Theorem~\ref{CyclicBypGroup}, whose statement is precisely the title of the section. We start by considering the following broader setting
\begin{quote}
$G$ is a finite group and $A$ is normal cyclic subgroup of $G$ such that $G/A$ is nilpotent
\end{quote}
rather than that of Theorem~\ref{CyclicBypGroup} since some results presented in this section will be used also in subsequent sections.

As $A$ is normal in $G$, so is $C_G(A)$. Moreover, as $G/A$ is nilpotent, so is $C_G(A)/A$. Furthermore, since $A$ is central in $C_G(A)$, the latter is nilpotent. Finally, as $\Aut(A)$ is abelian we conclude that $G/C_G(A)$ is abelian and hence $G'\subseteq C_G(A)$, i.e. $(G',A)=1$. We record these for future use:

\begin{lemma}\label{GmoduloCabeliano}
$C_G(A)$ is a normal nilpotent subgroup of $G$ containing $G'$. Thus, $(G', A)=1$.
\end{lemma}

The following lemma shows  some restrictions on a hypothetical minimal cyclic-by-nilpotent counterexample to (ZC) and on its possible negative partial augmentations.

\begin{lemma}\label{NegativeAugmentationAndOrder}
Suppose that (ZC) holds for every proper quotient of $G$ and let $u$ be a torsion element of $V(\Z G)$. Then

\begin{enumerate}
	\item If one of the following conditions holds for $x\in G$ then $\pa{u}{x}\ge 0$.
\begin{enumerate}
	\item\label{aumentofueraC} $x\not\in C_G(A)$. 
	\item\label{NotNormalSubgroup}  $(x,G)$ contains a non-trivial normal subgroup of $G$.
\end{enumerate}	
	\item\label{PrimesDividingA}  If $u$ is not rationally conjugate to any element of $G$ then every prime divisor of the order of  $C_G(A)$ divides the order of $u$.
\end{enumerate}

\end{lemma}

\begin{proof}
\eqref{aumentofueraC}
Let $x \in G\setminus C_G(A)$ and take $N=\GEN{(a,x^g) : a\in A, g\in G}$.
Then $N$ is a non-trivial normal subgroup of $G$ contained in $G'$.
By Remark~\ref{Formulaaumentoscociente}.\eqref{PAQuotientHyp} we have $\pa{u}{x^GN}\ge 0$.
On the other hand,  using $(G',A)=1$ (Lemma~\ref{GmoduloCabeliano}), we have that if $g,h\in G$ and $a\in A$ then
	\begin{align*}
	(a, x^h)x^g &
	= a^{-1} a ^{x^h} x^g
	= a^{-1} a ^{x^h} (x^h)^{(h^{-1}g)} (x^h)^{-1} x^h \\
	&= a^{-1} a ^{x^h} (h^{-1}g, (x^h)^{-1}) x^h
	= a ^{x^h} (h^{-1}g, (x^h)^{-1}) a^{-1} x^h
	= a ^{x^h} x^g (a ^{x^h})^{-1} \\
	& = x^{g (a ^{x^h})^{-1}} \in x^G.
	\end{align*}
This proves that $x^GN=x^G$. Therefore $\pa{u}{x}=\pa{u}{x^GN}\ge 0$, as desired.

\eqref{NotNormalSubgroup}
Let $N$ be a non-trivial normal subgroup of $G$ contained in $(x,G)$. 
Then for every $n\in N$ and $h\in G$  there is $g\in G$ with $n^{h^{-1}}=(x,g)$. 
Thus $x^hn = (x n^{h^{-1}})^h = (x (x ,g))^h = x^{gh}\in x^G$.
Therefore $\varepsilon_{x}(u)=\varepsilon_{Nx^G}(u)\ge 0$, by Remark~\ref{Formulaaumentoscociente}.\eqref{PAQuotientHyp}.

\eqref{PrimesDividingA}
Let $m=|u|$. By means of contradiction, let $p$ be a prime divisor of the order of $C_G(A)$ not dividing $m$. Let $D=Z(C_G(A))$. 
We claim that  $G$ contains a cyclic normal subgroup $B$ such that $A\subseteq B \subseteq D$ and $p\mid |B|$. 
This is clear if $p$ divides the order of $A$. 
So suppose that $p\nmid |A|$.
As $C_G(A)$ is nilpotent, $p$ divides the order of $D$. 
Then $D/A$ is a normal subgroup of $G/A$ of order divisible by $p$. 
As $G/A$ is nilpotent and $[D:A]$ is multiple of $p$, there is $d\in D$ such that $dA$ is a central element of order $p$ in $G/A$. Thus $B=\GEN{A,d}$ is a cyclic normal subgroup of $G$ of order divisible by $p$. This finishes the proof of the claim.
Replacing $A$ by $B$ we may assume without loss of generality that $p$ divides the order of $A$.

By Proposition~\ref{ZC1PAugmentation} there is $x\in G$ with $\varepsilon_x(u) < 0$ and by \eqref{aumentofueraC} we have that $x\in C_G(A)$. By Lemma~\ref{Hertweckaumentocero}, $\varepsilon_g(u)=0$ for every $g\in G$ with $g_p\ne 1$ and, in particular, $x_p=1$.
However, $A_p$ is a non-trivial normal subgroup of $A$ and hence $\sum_{X\in \Cl(G), X\subseteq x^GA_p} \pa{u}{X} = \pa{u}{x^G A_p}\ge 0$ by \eqref{PAQuotient} and Remark~\ref{Formulaaumentoscociente}.\eqref{PAQuotientHyp}. Hence $\pa{u}{g}>0$ for some $g\in x^GA_p$.
This leads to a contradiction since $(x,A)=1$ and hence $g_p\ne x_p = 1$.
\end{proof}

\begin{lemma}\label{ciclicopgrupo}
Suppose that $A$ is a Hall subgroup of $G$. Then
\begin{enumerate}
\item\label{CentralizadorAigualCentralizadorSocle} $C_G(A)=C_G(\Soc(A))$.
\item\label{CenterpElement}  If $Z(G)_p = 1$ for some prime integer $p$ then the Sylow $p$-subgroups of $G$ are abelian.
\end{enumerate}
\end{lemma}

\begin{proof}
By assumption $G=A\rtimes N$ where $A$ and $N$ are Hall subgroups of $G$ and $N$ is nilpotent.

\eqref{CentralizadorAigualCentralizadorSocle} If $p$ and $q$ are prime integers with $p\mid |A|$ and $q\mid |N|$
then $q\ne p$. This implies that $C_{N_q}(A_p)=C_{N_q}(\Soc(A_p))$.
Thus
    $$C_N(A)=\prod_{q\mid |N|}\bigcap_{p\mid |A|} C_{N_q}(A_p)=
    \prod_{q\mid |N|} \bigcap_{p\mid |A|} C_{N_q}(\Soc(A_p))=C_N(\Soc(A))$$
and therefore $C_G(A)=A\rtimes C_N(A)=A\rtimes C_N(\Soc(A))=C_G(\Soc(A))$.

\eqref{CenterpElement} Suppose that one Sylow $p$-subgroup of $G$ is non-abelian. Then $N_p$ is non-abelian. Consider the homomorphism $\alpha : N_p \rightarrow \Aut(A)$ mapping $g\in N_p$ to the automorphism of $A$ given by conjugation by $g$.
Since $\Aut(A)$ is abelian, $\ker(\alpha)$ is a non-trivial normal subgroup of $N_p$. As $N$ is nilpotent, $1\ne \ker(\alpha) \cap Z(N_p) \subseteq Z(G)_p$.
\end{proof}

\begin{proofof} {\it Theorem \ref{CyclicBypGroup}}. Assume that $p$ is a prime integer, $G$ is a finite group and $A$ is a normal cyclic subgroup of $G$ such that the factor group $G/A$ is a $p$-group. We may assume without loss of generality that $A$ is a Hall $p'$-subgroup.
Then $G=A\rtimes P$ where $P$ is a Sylow $p$-subgroup of $G$.
Furthermore, by means of contradiction, we assume that $G$ is a counterexample of minimal order to the theorem and that $u$ is a torsion unit of $V(\Z G)$ of minimal order which is not rationally conjugate to any element of $G$.
Hence the Zassenhaus Conjecture holds for every proper  quotient of $G$.
By Proposition~\ref{ZC1PAugmentation}, there exists $x\in G$ such that  $\pa{u}{x}<0$.
As $C_G(A)$ has at most one non-cyclic Sylow subgroup, by Theorem~\ref{G'OneNonCyclic}, $u\not\in V(\Z G,C_G(A))$ i.e. $\omega_{C_G(A)}(u)\ne 1$.
Furthermore, by Lemma~\ref{NegativeAugmentationAndOrder}.\eqref{aumentofueraC} it follows that $x$ lies in $C_G(A)$.

Set $m=|u|$. By Lemma~\ref{NegativeAugmentationAndOrder}.\eqref{PrimesDividingA}, $m$ is multiple of every prime dividing the order of $A$. Moreover, $p$ divides $m$, since $\omega_{C_G(A)}(u)\ne 1$. Thus every prime divisor of the order of $G$ divides $m$. It is well-known that (ZC) holds for cyclic-by-abelian groups \cite{CaicedoMargolisdelRio2013}. Thus $P$ is not abelian.
By Lemma~\ref{ciclicopgrupo}.\eqref{CenterpElement}
there exists a central element $z$ of $G$ of order $p$.
Then, by the induction hypothesis $\omega_{\GEN{z}}(u)$ is conjugate to
$\omega_{\GEN{z}}(g)$ in the units of $\Q (G/\GEN{z})$, for some
$g\in G$.
As $1\neq \omega_{C_G(A)}(u)= \omega_{C_G(A)/\GEN{z}}(\omega_{\GEN{z}}(u))$ we have $\omega_{\GEN{z}}(g)\notin C_G(A)/\GEN{z}$. Hence $g\notin C_G(A)$. On the other hand,
by the above paragraph the order of $\omega_{\GEN{z}}(u)$, $\omega_{\GEN{z}}(g)$ and $g$ are divisible by all the prime divisors of the order of $A$. Then $\Soc(A)\subseteq \GEN{g_{p'}}\subset \GEN{g}$ and hence
$C_G(g)\subseteq C_G(g_{p'})\subseteq C_G(\Soc(A))=C_G(A)$, by  Lemma~\ref{ciclicopgrupo}.\eqref{CentralizadorAigualCentralizadorSocle}. Thus, $g$ lies in $C_G(A)$, a contradiction.
This finishes the proof of Theorem~\ref{CyclicBypGroup}.
\end{proofof}

\section{A formula for partial augmentations}\label{SectionFormulaPA}

In this section we revisit a formula from \cite{CaicedoMargolisdelRio2013}.

For any abelian normal subgroup $N$ of $G$ let
$$\K_N = \left\{K\leq N  :
\matriz{{l} N/K  \text{ is cyclic and } K \text{ does not contain}  \\
	\text{any non-trivial normal subgroup of } G
} \right\}.$$

We start with the following lemma, which is a generalization of Remark 3.2 in \cite{CaicedoMargolisdelRio2013}.
It describes $\K_N$, in the cases we need.  

\begin{lemma}\label{FormaK}
	Let $G$ be a finite group containing a normal cyclic subgroup $A$ such that $G/A$ is nilpotent. Let $N$ be an abelian normal subgroup of $G$ containing $A$. Then
	$$\K_N = \{K\leq N : K\cap A=K\cap Z(G) =1 \hspace{0.2cm} \text{and} \hspace{0.2cm} N/K \hspace{0.2cm} \text{is cyclic}\}$$
	and for every $K\in \K_N$  we have
	$$|\K_N|\le |K|=\frac{|N|}{\Exp(N)}.$$
\end{lemma}

\begin{proof}
	Every subgroup of $A$ and every subgroup of $Z(G)$ is normal in $G$. 
	Thus, if $K\in \K_N$ then $K\cap A=K\cap Z(G)=1$. 
	This proves one inclusion of the first equality.
	Conversely, let $K$ be a subgroup of $N$ containing a non-trivial normal subgroup $X$ of $G$ and such that $K\cap A = 1$.
	Let us use bar notation for reduction modulo $A$.
	Observe that $\overline{X} \cap Z(\overline{G}) \neq \bar{1}$ since $\overline{X}$ is a non-trivial normal subgroup of $\overline{G}$, and the latter is nilpotent.
	Let $x\in X$ with $\overline{x}\in Z(\overline{G})\setminus \{1\}$. 
	Then $(x,g)\in A$ for every  $g\in G$. 
	On the other hand, as $X$ is a normal subgroup of $G$, $(x,g)\in A\cap X \subseteq A\cap K = 1$.
	Thus $x$ is a central element of $G$ and therefore 
	$1\ne x \in  K\cap Z(G)$.
	This proves the first equality.
	
	Let $\LL_N=\{L\le N : [N:L]=\Exp(N) \text{ and } N/L \text{ is cyclic}\}$.
	To finish the proof we will show that $\K_N\subseteq \LL_N$ and $|\LL_N|\le \frac{|N|}{\Exp(N)}$.
	Indeed, suppose that $K\in \K_N$.  It is well known that there is an involution $\alpha$ of the lattice of subgroups of $N$ such that $N/H\cong \alpha(H)$ for every $H\le N$.
	In particular, $\alpha(K)$ is a cyclic subgroup of $N$ of order $[N:K]$ and hence $[N:K]\le \Exp(N)$.
	Suppose that $[N:K]<\Exp(N)$ and let $X=N^{[N:K]}$.
	Then $1\ne X \le K$.
	As $N$ is normal in $G$ and $X$ is a characteristic subgroup of $N$ we deduce that $X$ is normal in $G$.
	Thus $K$ contains a non-trivial normal subgroup of $G$, contradicting the hypothesis.
	Thus $[N:K]=\Exp(N)$, i.e. $K\in \LL_N$.
	Using the involution $\alpha$ we see that the cardinality of $\LL_N$ coincides with the cardinality of the set $Y_N$ of cyclic subgroups of $N$ of order $\Exp(N)$.
	So it remains to prove that $|Y_N|\le \frac{N}{\Exp(N)}$.
	For that we may assume without loss of generality that $N$ is a $p$-group because if $N=P\times Q$ with $P$ and $Q$ of coprime order then the map $(K,L)\mapsto K\times L$ defines a bijection from $Y_P\times Y_Q$ to $Y_N$.
	So suppose that $N$ is a $p$-group and let $p^e$ be the exponent of $N$.
	Then $N=P\times Q$ with $P\cong C_{p^e}^l$, for some $l\ge 1$ and $Q$ of exponent smaller than $p^e$.
	Then the elements of $N$ generating a cyclic subgroup of order strictly smaller than $p^e$ are precisely the elements of $P^p\times Q$. Therefore $N$ has exactly $|N|-p^{l(e-1)}|Q|$ elements of order $p^e$ and hence
	$|Y_N|=\frac{|N|-p^{l(e-1)}|Q|}{(p-1)p^{e-1}} = \frac{|N|}{p^e}\frac{p-p^{1-l}}{p-1}\le \frac{|N|}{p^e}$, because $p^{1-l}\le 1$.
\end{proof}

Given $g,h\in G$ and a subgroup $K$ of $G$ let 
$$X_{K,g,h}=\{t\in G : g^t \in h K\}$$
and 
$$Y_{K,g,h}= \{k \in K : k = (g^h,h^{-1}x) \text{ for some }
x\in G\}.$$

The following lemma will be used in Section~\ref{SectionCyclicByHamiltonian}.

\begin{lemma}\label{CardinalXs}
	Let $g\in G$, $x\in X_{K,g,h}$. 
	Then $|X_{K,g,h}|$ is a multiple of $|C_G(g)|$ and 	$|X_{K,g,h} | \le |C_G(g)|\; |Y_{K,g,x}|$.
\end{lemma}

\begin{proof}
Clearly $C_G(g)X_{K,g,h}\subseteq X_{K,g,h}$ and therefore $X_{K,g,h}$ is a union of right $C_G(g)$-cosets and, in particular, $|X_{K,g,h}|$ is a multiple of $|C_G(g)|$. 
	
	Let $x,y\in X_{K,g,h}$. Then 
	$$(g^x,x^{-1}y) = x^{-1} g^{-1} x y^{-1} g y = (g^x)^{-1} g^y \in K.$$
	Therefore $(g^x,x^{-1}y)\in Y_{K,g,x}$. 
	Moreover, if $z$ is another element of $X_{K,g,h}$ then $(g^x,x^{-1}y)=(g^x,x^{-1}z)$ if and only if $g^y=g^z$ if and only if $yz^{-1} \in C_G(g)$. 
	This proves that $C_G(g)y\mapsto (g,xy^{-1})$ defines an injective map from the set of right $C_G(g)$-cosets contained in $X_{K,g,h}$ to $Y_{K,g,x}$. Thus $|X_{K,g,h} | \le |C_G(g)|\; |Y_{K,g,x}|$.
\end{proof}

For a square matrix $U$ with entries in $\C$ and $\alpha \in \C$, let $\mu_U(\alpha)$ denote the multiplicity of $\alpha$ as eigenvalue of $U$.

For an abelian normal subgroup of $G$ and every $K\in \K_N$, we select a linear character $\psi_K$ of $N$ with kernel $K$ and we fix a representation $\rho_K$ of $G$ affording the induced character
$\psi_K^G$. Observe that if $K_1$ and $K_2$ are conjugate in $G$, then $\psi_{K_1}^G=\psi_{K_2}^G$ and therefore we may assume $\rho_{K_1}=\rho_{K_2}$.

\begin{proposition}\label{PartialAugmentationTraces}
	Let $G$ be a finite group such that (ZC) holds for every proper quotient of $G$.
	Let $N$ be an abelian normal subgroup of $G$, let $\K=\K_N$ and let $\mathcal{K}$ be a set of representatives of the $G$-conjugacy classes of the elements of $\;\K$ .
	Let $u$ be  an element of order $m$ in $V(\Z G)\setminus V(\Z G,N)$ such that every proper power of $u$ is rationally conjugate to an element of $G$. 
	Let $x\in N$ with $x^m=1$ and let $f$ be a positive integer such that $u^f$ is rationally conjugate to an element $\y$ of $N$. Then
	\begin{align}\label{KsPartialAugmentation}
	\begin{split}
	&\sum_{K\in \K} \varphi([N:K]) \; \mu_{\rho_K(u)}(\psi_K(x)) = \\
	& \hspace{1cm} \frac{\varphi(m)}{m}  \; |C_G(x)|  \; \varepsilon_{x}(u) +
	\frac{[C_G(\y):N]}{f} \;\sum_{K\in \mathcal{K}} \frac{\varphi([N:K])}{|N_G(K)|} \sum_{z\in \y^G} |X_{K,x^f,z}|.
	\end{split}
	\end{align}
\end{proposition}

\begin{proof}
	By \cite[Lemma~3.1]{CaicedoMargolisdelRio2013} we have 
	\begin{align}\label{KsPartialAugmentationOriginal}
	\begin{split}
	&\sum_{K\in \K} \varphi([N:K]) \; \mu_{\rho_K(u)}(\psi_K(x)) = \\
	& \hspace{1cm} \frac{\varphi(m)}{m}  \; |C_G(x)|  \; \varepsilon_{x}(u) +
	\frac{1}{f}\sum_{K\in \K} \varphi([N:K]) \;  \mu_{\rho_K(u^f)}(\psi_K(x^f)).
	\end{split}
	\end{align}
	
	Choose a transversal $T$ of $N$ in $G$.
	Observe that $N\subseteq C_G(\y)$, as $N$ is abelian.
	Thus $|\{gN\in G/N : x^f \in \y^gK\}|=[C_G(\y):N]\;|x^f K\cap \y^G|$.
	As $u^f$ is conjugate to $\y$ in $\Q G$, $\rho_K(u^f)$, $\rho_K(\y)$ and $\diag(\psi_K(\y^g) : g\in T)$ are conjugated as complex matrices.
	Thus
	\begin{equation*}\begin{split}
	\mu_{\rho_K(u^f)}(\psi_K(x^f))& =|\{gN\in G/N : \psi_K(x^f)=\psi_K(\y^g)\}|
	= |\{gN\in G/N : \y^g \in x^f K\}| \\
	& =[C_G(\y):N]\;|x^fK\cap \y^G|.
	\end{split}\end{equation*}
	Therefore 
	\begin{equation}\label{MuCardinal}
	\begin{split}
	&\sum_{K\in \K} \varphi([N:K]) \;  \mu_{\rho_K(u^f)}(\psi_K(x^f)) 
	= \sum_{K\in \mathcal{K}} \frac{\varphi([N:K])}{|N_G(K)|} 
	\sum_{t\in G} \mu_{\rho_{K^t}(u^f)}(\psi_{K^t}(x^f)) \\
	&\hspace{3cm}  = [C_G(\y):N]\; 
	\sum_{K\in \mathcal{K}} \frac{\varphi([N:K])}{|N_G(K)|} \sum_{t\in G} |(x^f)^t K \cap \y^G|.
	\end{split}
	\end{equation}
	On the other hand,
	\begin{equation}\label{SumaCardinales}
	\sum_{t\in G} |(x^f)^t K \cap \y^G| = |\{(t,z)\in G\times \y^G : (x^f)^t \in zK \} |
	= \sum_{z\in \y^G} |X_{K,x^f,z}|.
	\end{equation}
	The lemma follows from \eqref{KsPartialAugmentationOriginal}, \eqref{MuCardinal} and \eqref{SumaCardinales}.
\end{proof}

\section{On the Zassenhaus Conjeture for cyclic-by-Hamiltonian groups}\label{SectionCyclicByHamiltonian}

In this section we investigate the Zassenhaus Conjecture for cyclic-by-Hamiltonian groups. 
In the first part of the section we do not consider any additional hypothesis. After Lemma~\ref{aumentofueraD} we consider a hypothetical minimal counterexample to (ZC) in the class of cyclic-by-Hamiltonian groups and prove some features of it. 
This is used to prove Theorem~\ref{CyclicByHamiltonian} at the end of the section.

So  throughout 

\begin{quote} 
$G$ is a finite group and $A$ is a cyclic normal subgroup of $G$ such that $G/A$ is Hamiltonian.
\end{quote}

We also fix a generator $a$ of $A$ and set 
    $$n=|A|,\quad C=C_G(A) \qand D=Z(C).$$ 
Furthermore, we assume that $G$ is not cyclic-by-abelian, for otherwise (ZC) holds for $G$ by the main result of \cite{CaicedoMargolisdelRio2013}. As a consequence, $n$ must be even. By Lemma~\ref{GmoduloCabeliano}, $C$ is nilpotent and contains $G'$. Thus  $|C|$, $|D|$ and $\Exp(D)$ are divisible by the same primes and, as $G/A$ is Hamiltonian, we have $G'\subseteq A\times \GEN{\nu}$ with $\nu$ an element of $G'\setminus A$ of order $2$.

Since $G'_{2'}$ is cyclic and $D$ is abelian, Theorem~\ref{G'OneNonCyclic} implies that

\begin{proposition}\label{ZCtrueforD}
 Every torsion element of $V(\Z G,D)$ is rationally conjugate to an element of $D$.
\end{proposition}

\begin{lemma}\label{CommutatorOdd}
	If $g$ is a $2'$-element of $G$ then $(g,G)\subseteq A_{2'}$ and $(g,D_2)=1$.
\end{lemma}

\begin{proof}
	Let $g,h\in G$, $b=(g,h)$ and $k=|g|$ with $k$ odd. 
	Then $b\in A$ and $b^h=b^r$ for some integer $r$.
	Therefore $h=h^{g^k} = hb^{1+r+r^2+\dots r^{k-1}}$ and hence $|b|$ divides $1+r+r^2+\dots+r^{k-1}$. 
	So, if $b\not\in A_{2'}$ then $r$ is odd and $1+r+r^2+\dots+r^{k-1}$ is even and hence $0\equiv 1+r+\dots+r^{k-1} \equiv k \equiv 1 \mod 2$, a contradiction.
	This finishes the proof of the the first statement.
	The second is consequence of the first because $D_2$ is a normal in $G$ and hence $(g,D_2)= D_2\cap A_{2'}=1$.
\end{proof}

We now deal with some partial augmentations.

\begin{lemma}\label{aumentofueraD}
Suppose that (ZC) holds for proper quotients of $G$.
If $u$ is a torsion element of $V(\Z G)$ and $g \in G \setminus D$, then $\pa{u}{g} \ge 0$.
\end{lemma}

\begin{proof}
Let $C_1=C_G(A,\nu)$.
Let $x,y\in G$ such that $(x,y)\not\in A$. 
As $G/A$ is Hamiltonian, $\GEN{A,\nu}=\GEN{A,(x,y)}=\GEN{A,x_2^2}\subseteq \GEN{A,x^2}$.
Therefore $\nu\in \GEN{x^2,A}$ and hence, if moreover $x\in C$ then $x\in C_1$.

\textbf{Claim}: If there is $x\in G$ such that  $\GEN{(g,x)}$ is  non-trivial and normal in $G$ and $((g,x),g)=((g,G),x)=1$ then $\pa{u}{g}\ge 0$.

Indeed, by assumption, for every $h\in G$ we have
$$(g,x)g^h =(g,x)g(g,h)=g(g,x)(g,h)=g(g,x)(g,h)^x = g(g,hx)=g^{hx}.$$
Thus $g^G N=g^G$ for $N=\GEN{(g,x)}$. Therefore $\pa{u}{g}=\pa{u}{g^GN}\ge 0$, as desired.

Suppose that $g\in G\setminus D$.
If $g\not\in C$ then $\pa{u}{g}\ge 0$ by Lemma~\ref{NegativeAugmentationAndOrder}.\eqref{aumentofueraC}.
Thus we may assume that $g\in C$.
If $(g,\nu)\ne 1$ then $(g,\nu)=a^{\frac{n}{2}}$ and hence $\GEN{a^{\frac{n}{2}}}$ is a non-trivial normal subgroup of $G$ contained in $(g,G)$. Then $\pa{u}{g}\ge 0$ by Lemma~\ref{NegativeAugmentationAndOrder}.\eqref{NotNormalSubgroup}.
Thus we may assume that $(g,\nu)=1$.
Suppose that $(g,y^2)\ne 1$ for some $y\in C$.
Then $y^2\in C_1$ and $(g,y^2)=(g,y)(g,y)^y\in A$, since $[G':G'\cap A]\le 2$.
Thus $g$ and $x=y^2$ satisfy the conditions of the claim and hence $\pa{u}{g}\ge 0$.
Thus as well as $(g,\nu)=1$ we may also assume that $(g,x^2)=1$ for every $x\in C$.
Therefore $C/C_C(g)$ is an elementary abelian 2-group and hence $|g^C|$ is a power of $2$.
Moreover, $(g,x)(g,x)^x=1$, i.e. $(g,x)^x=(g,x)^{-1}$ for every $x\in C$.
If $(g,x)\in A$ then $(g,x)=(g,x)^x=(g,x)^{-1}$ and hence $(g,x)\in \GEN{a^{\frac{n}{2}}}$.
Therefore, if $(g,x)\in A\setminus \{1\}$ then $\GEN{(g,x)}$ is a non-trivial normal subgroup of $G$ contained in $(g,G)$. In such case $\pa{u}{g}\ge 0$, by Lemma~\ref{NegativeAugmentationAndOrder}.\eqref{NotNormalSubgroup}.  
Therefore we may assume that $(g,x)\not\in A$ for every $x\in C\setminus C_C(g)$. 

Then for every $x\in C\setminus C_C(g)$ we have $(g,x)=a^i\nu$ for some $i$ and hence $a^i \nu^x = (g,x)^x = (g,x)^{-1}=a^{-i}\nu$. 
Therefore $\nu^x=\nu$ and $a^i=a^{-i}$. Thus $a^i \in\GEN{a^{\frac{n}{2}}}$ and $(g,x)\in \{\nu,a^{\frac{n}{2}}\nu\}$. By the first paragraph of the proof we have $C\setminus C_C(g)\subseteq C_1$ and $g^C\subseteq \{g,\nu g,a^{\frac{n}{2}}\nu g\}$.  As $|g^C|$ is a power of $2$ we deduce that $g^C$ is either $\{g,\nu g\}$ or $\{g, a^{\frac{n}{2}}\nu g\}$. Replacing $\nu$ by $a^{\frac{n}{2}}\nu$ in the second case, one may assume that $g^C=\{g,\nu g\}$. Fix $x\in C$ with $g^x=\nu g$. Moreover, as $1\in C_1\cap C_C(g)$, we have that $C\setminus C_C(g)$ is properly included in $C_1$ and hence $C_1$ is a subgroup of $C$ with $|C_1|>\frac{|C|}{2}$. Thus $C=C_1$.

As $G/\GEN{A,\nu}$ is abelian, for every $h\in G$ we have $g^h = ug$ and $x^h = vx$ with $u,v\in \GEN{A,\nu}$. Then $(u,g)=(u,x)=(v,g)=(v,x)=(u,v)=1$ and therefore $\nu^h = (g^h,x^h)=(ug,vx)=(g,x)=\nu$.
Thus $\nu=(g,x)$ is central in $G$ and hence $g$ and $x$ satisfies the conditions of the claim, and once more $\pa{u}{g}\ge 0$.
\end{proof}

\begin{quote}
In the remainder we suppose that (ZC) holds for proper quotients of $G$ but not for $G$. Therefore, there is an element $u\in \Z G$ such that $u$ is not rationally conjugate to any element of $G$ and we assume that the order of $u$ is minimal with this property.
\end{quote}

We fix the following notation: 
$$ m=|u|, \quad f=|\omega_D(u)|, \quad \K=\K_D \qand B=\GEN{a^{\gcd(f,n)}}.$$

Observe that $B$ is the only subgroup of $A$ of order $\gcd(f,n)$.
As $G'\subseteq A\times \GEN{\nu}\subseteq D$ and $f$ divides the exponent of $G/D$, we deduce that $f$ divides the exponent of $G/G'$, which is not multiple of $4$. Thus $4\nmid f$. 
By Proposition~\ref{ZCtrueforD}, $f>1$ and, by Proposition~\ref{ZC1PAugmentation} and Lemma~\ref{aumentofueraD}, there are 
	$$\x,\y \in D, \text{ with }
	\varepsilon_{\x}(u)<0 \text{ and } u^f  \text{ rationally conjugate to } \y.$$
Write

\begin{quote}
	$f=f'f''$, with  
	$\gcd(f'',|D|)=1$ and the prime divisors of $f'$ dividing $|D|$.
\end{quote} 
Then, by Lemma~\ref{NegativeAugmentationAndOrder}.\eqref{PrimesDividingA}, 
\begin{quote}
	$m=m'f''$, with $m'$ and $\Exp(D)$ having the same prime divisors.
\end{quote}
In particular, 
\begin{equation}\label{FraccionIgual}
\frac{\varphi(\Exp(D))}{\Exp(D)} = \frac{\varphi(m')}{m'}.
\end{equation}

Applying \eqref{KsPartialAugmentation} to $N=D$ and $x=\x$ we have 

\begin{align*}
\begin{split}
&\frac{[C_G(\gamma):D]}{f} \;\sum_{K\in \mathcal{K}} \frac{\varphi([D:K])}{|N_G(K)|} \sum_{z\in \gamma^G}|X_{K,\x^f,z}|= \sum_{K\in \K} \varphi([D:K]) \; \mu_{\rho_K(u)}(\psi_K(\x)) \\
&- \frac{\varphi(m)}{m}  \; |C_G(\x)|  \; \varepsilon_{\x}(u) >0,
\end{split}
\end{align*}
which implies that $\K\ne\emptyset$.
Moreover,  if $x$ is an element of $K\cap (A\times \GEN{\nu})$ with $K\in \K$ then $x^2\in K\cap A=1$ and $x\ne a^{\frac{n}{2}}$. Also $K\cap (A\times \GEN{\nu})\ne 1$ because $D/K$ is cyclic. This proves the following:

\begin{equation}\label{KIntersecionAv}
\text{For every } K\in \K, \text{ we have } K\cap G' = K\cap (A\times \GEN{\nu}) = 
\begin{cases}
\GEN{\nu},  \\ \text{ or } \\
\GEN{a^{\frac{n}{2}}\nu}.
\end{cases}
\end{equation}
It follows that $\nu\not\in Z(G)$ and hence $\nu^G=\{\nu,a^{\frac{n}{2}}\nu\}$.  Moreover, for every $g,h\in G$ we have  $Y_{K,g,h}\subseteq K\cap G'$  and hence $|Y_{K,g,h}|\le 2$. Thus, by Lemma~\ref{CardinalXs}, 
$|X_{K,g,h}|$ is either $0$, $|C_G(g)|$ or $2|C_G(g)|$. 
In particular 
\begin{equation}\label{CotaGeneral}
|X_{K,\x^f,z}|\le 2|C_G(\x^f)|, \quad \text{for every }z\in G.
\end{equation}
Furthermore, if $(\x^f,G)\cap \nu^G=\emptyset$ then $Y_{K,\x^f,z}\subseteq K\cap A=1$. If $2\mid f$ then $(\x^f,G)\subseteq A$ and therefore 
\begin{equation}\label{CotaEspecial}
\text{if } 2\mid f  \text{ or } (\x^f,G)\cap \nu^G = \emptyset \text{ then } |X_{K,\x^f,z}|\le |C_G(\x^f)|, \quad \text{for every }z\in G.
\end{equation}
Combining  \eqref{CotaGeneral} and \eqref{CotaEspecial} we obtain 
\begin{equation}\label{CotaX}
|X_{K,\x^f,z}| \le \frac{2}{\gcd(2,f)} \; |C_G(\x^f)| = \frac{2}{\gcd(2,f')} \; |C_G(\x^f)|.
\end{equation}

As $\x\in D$, we have $\GEN{\x^f}=\GEN{\x^{f'}}$ and therefore $C_G(\x^f)=C_G(\x^{f'})$. 
Thus, if $g\in C_G(\x^f)$ then $(\x,g)^{f'} = (\x^{f'},g)=1$ and so $(\x,g)\in B\times \GEN{\nu}$.  
Hence $(\x,C_G(\x^f))\subseteq B\times \GEN{\nu}$ and if $f$ is odd then $(\x,C_G(\x^f))\subseteq B\subseteq A$. This proves that:

\begin{equation}\label{IndexCxfCx}
[C_G(\x^f):C_G(\x)]\le \gcd(2,f)\; \gcd(f,n) = \gcd(2,f')\; \gcd(f',n). 
\end{equation}
Also, if $\gcd(f,n)\ne 1$ then, by Lemma~\ref{NegativeAugmentationAndOrder}.\eqref{NotNormalSubgroup}, $B\not\subseteq (\x,C_G(\x^f))$.
Thus, 
\begin{equation}\label{IndexCxfCxStrict}
\text{if } \gcd(f,n)\ne 1 \text{ then } [C_G(\x^f):C_G(\x)]< \gcd(2,f)\; \gcd(f,n).
\end{equation}

On  one hand, by Lemma~\ref{FormaK}, for every $K\in \K$ we have $[D:K]=\Exp(D)$. Combining this with \eqref{CotaX}, and then using $D\subseteq C_G(\gamma)$ and the inequality $|\K|\le \frac{|D|}{\Exp(D)}$ from Lemma~\ref{FormaK}, and \eqref{FraccionIgual} we obtain
\begin{align}\label{CuentaLarga}
\begin{split}
&\frac{[C_G(\y):D]}{f} \;\sum_{K\in \mathcal{K}} \frac{\varphi([D:K])}{|N_G(K)|} \sum_{z\in \y^G} |X_{K,\x^f,z}| \\
&  \le \frac{2\; [C_G(\y):D]\; \varphi(\Exp(D))\;
	|\y^G|\; |C_G(\x^f)|}{f\; \gcd(2,f')}  \;\sum_{K\in \mathcal{K}} \frac{1}{|N_G(K)|}  \\
&  = \frac{2\; [C_G(\y):D]\;\varphi(\Exp(D))\;
	[G:C_G(\y)]\; |C_G(\x^f)|}{f\; \gcd(2,f')}  \;\sum_{K\in \mathcal{K}} \frac{1}{|N_G(K)|}  \\
 &  \le \frac{2}{f\; \gcd(2,f')} \; \frac{\varphi(\Exp(D))\;
 |C_G(\x^f)|}{|D|} 
 \;\sum_{K\in \mathcal{K}} [G:N_G(K)]  \\
 &  = \frac{2}{f\; \gcd(2,f')} \; 
 \frac{\varphi(\Exp(D))\; |C_G(\x^f)|}{|D|} \;|\K|  \\ 
&  \le \frac{2}{f\; \gcd(2,f')} \; 
\frac{\varphi(\Exp(D))\; |C_G(\x^f)|}{\Exp(D)}   \\
&= \frac{2}{f\; \gcd(2,f')} \; 
\frac{\varphi(m')\; |C_G(\x^f)|}{m'}   \\ 
 &  = \frac{2}{f'\; \gcd(2,f')\;\varphi(f'')} \; 
 \frac{\varphi(f'')\varphi(m')\; |C_G(\x^f)|}{f''\; m'}   \\
 & = \frac{2}{f'\; \gcd(2,f')\;\varphi(f'')} \; 
 \frac{\varphi(m)\; |C_G(\x^f)|}{m}.
\end{split}
\end{align}

On the other hand, the left side of the equality \eqref{KsPartialAugmentation} is  non-negative and $\varepsilon_{\x}(u)\le -1$. This implies that 
\begin{eqnarray}\label{CuentaCorta}
\begin{split}	
\frac{[C_G(\y):D]}{f} \;\sum_{K\in \mathcal{K}} \frac{\varphi([D:K])}{|N_G(K)|} \sum_{z\in \y^G} |X_{K,\x^f,z}|&\ge - \frac{\varphi(m)\; |C_G(\x)|}{m} \varepsilon_{\x}(u) 
\\ &\ge \frac{\varphi(m)\; |C_G(\x)|}{m}.
\end{split}
\end{eqnarray}
Combining \eqref{CuentaLarga}, \eqref{CuentaCorta} and \eqref{IndexCxfCx} we obtain 
\begin{equation}\label{Emparedado}
\frac{f'\;\gcd(2,f')\;\varphi(f'')}{2} \le [C_G(\x^f):C_G(\x)] \le \gcd(2,f')\; \gcd(f',n).
\end{equation}
Thus $\frac{f'}{2}\;\varphi(f'') \le \gcd(f',n)$. In particular, $\varphi(f'')\le 2$ and $\frac{f'}{2}\le \gcd(f',n)$.
The latter implies that if $f'\nmid n$ then $\gcd(f',n)=\frac{f'}{2}$. In this case $f'$ is even and hence $2\mid \gcd(f,n)$. Thus the second inequality in \eqref{Emparedado} is strict, by \eqref{IndexCxfCxStrict}, and then $\frac{f'}{2} < \gcd(f',n)=\frac{f'}{2}$, a contradiction.
Therefore 
	$$\varphi(f'')\le 2, \quad f'\mid n  \text{ so that } |B|=f' \quad \text{ and } \quad \text{if } f'\ne 1 \text{ then } f'' =1.$$

We finish our analyses with the following lemma.

\begin{lemma}\label{Casos}
\begin{enumerate}
	\item\label{fEven} If $f$ is even then $f''=1$, $(\x^f,G)\subseteq A$ and $(\x,C_G(\x^f))\not\subseteq A$.
	\item\label{fOdd} If $f$ is odd then $(\x,C_G(\x^f))\subseteq B$ and $(\x^f,G)\cap \nu^G \ne \emptyset$. Furthermore either $f=f''=3$ or $f''=1$.
	\item\label{(delta,G)}
	$(\x,G)\not\subseteq A$.
	\item\label{nuenD} $G'\subseteq A\times \GEN{\nu}\subseteq D$ and $4\nmid f$.
	\item\label{NotCentralOrder2}  $\nu^G\subseteq (D,G)$.
\end{enumerate}
\end{lemma}

\begin{proof}
We first observe that by Lemma~\ref{NegativeAugmentationAndOrder}.\eqref{NotNormalSubgroup}, $B\times \GEN{\nu}\not\subseteq (\x,C_G(\x^f))$ and if $f'\ne 1$ then $B\not\subseteq (\x,C_G(\x^f))$.
	
\eqref{fEven} If $f$ is even then $(\x^f,G)\subseteq A$ and $f'\ne 1$.
Then $f''\le 2$ and, as $\gcd(f',f'')=1$, we deduce that $f''=1$.
Thus  $|B|=f\ne 1$ and hence  $B\not\subseteq (\x,C_G(\x^f))$ by Lemma~\ref{NegativeAugmentationAndOrder}.\eqref{NotNormalSubgroup}.  
By \eqref{Emparedado}
	$$f\le [C_G(\x^f):C_G(\x)]=|\x^{C_G(\x^f)}| = |(\x,C_G(\x^f))|.$$
Therefore $(\x,C_G(\x^f))\not\subseteq A$ for otherwise $(\x,G)$ contains $B$. 
	
\eqref{fOdd} If $f$ is odd then, $(\x,C_G(\x^f))\subseteq B$. 
If $f'=1$ then $f''\ne 1$ and hence $\varphi(f)=\varphi(f'')=2$. Then $f''=3$.

We now prove that $(\x^f,G)\cap \GEN{\nu}\ne \emptyset$.
Otherwise
$|X_{K,\x^f,z}|\le |C_G(\x^f)|$ for every $K\in \K$ and $z\in G$, by \eqref{CotaEspecial}. 
Using this in the calculations in \eqref{CuentaLarga} we get $f'\varphi(f'')\le [C_G(\x^f):C_G(\x)]\le f'$. Therefore $\varphi(f'')=1$ and $[C_G(\x^f):C_G(\x)]=f'$. 
As $f$ is odd this implies that $f''=1$. 
Therefore $1\ne f'=\gcd(f,n)$ and hence $f=[C_G(\x^f):C_G(\x)]<f'$, by \eqref{IndexCxfCxStrict}.
This yields a contradiction and finishes the proof of \eqref{fOdd}. 

\eqref{(delta,G)} If $f$ is even then \eqref{(delta,G)} follows from \eqref{fEven}. 
So, if $(\x,G)\subseteq A$ then $f$ is odd. Moreover, for every $g\in G$ we have $(\x^f,g)=(\x,g)^f \in A$ in contradiction with \eqref{fOdd}. 
This finishes the proof of \eqref{(delta,G)}. 

\eqref{nuenD} By \eqref{(delta,G)}, $G'\subseteq A\times \GEN{\nu} \subseteq \GEN{A,(\x,G)} \subseteq D$. Then as $f$ divides $\Exp(G/D)$, it also divides  $\Exp(G/G')$ and the latter is not multiple of $4$. So $4\nmid f$.

\eqref{NotCentralOrder2} 
Suppose that $(D,G)\cap \nu^G= \emptyset$. 
In particular, $(\x^f,G)\cap \nu^G=\emptyset$ and hence $f$ is even by \eqref{fOdd}.
Then  $(\x,C_G(\x^f))\nsubseteq A$ by \eqref{fEven}. 
Let $g\in C_G(\x^f)$ with $(\x,g)\not\in A$. 
Since $(\x_{2'},g)\in A$, $1=(\x_2^f,g)=(\x_2,g)^f$ and $4\nmid f$, by \eqref{(delta,G)}, we have $(\x_2,g)^2=1$.
Hence $(\x_2,g)\in \Soc(G'_2)\setminus A  = \nu^G$, in  contradiction with the assumption. This proves that there are $x\in D$ and $y\in G$ such that 
$(x,y)$ is either $\nu$ or $a^{\frac{n}{2}}\nu$.
As $\nu$ is not central in $G$ there is $g\in G\setminus C_G(\nu)$.
Then $a^{\frac{n}{2}}(x,y)=(x,y)^g=(x^g,y^g)$, which is the other element of $\nu^G$.
Thus $\nu^G\subseteq (D,G)$.
\end{proof}

%

\begin{proofof}{\it Proposition~\ref{MinimalCounterexampleHamiltonian}}.
	By Lemma~\ref{Casos}.\eqref{NotCentralOrder2} there are $x\in D$ and $y\in G$ such that $(x,y)=\nu$. Then $(x_2,y_2)=\nu$ and hence one may assume that $x$ and $y$ are $2$-elements. As $G/A$ is Hamiltonian it easily follows that $\GEN{x,y,A}/A$ is isomorphic to the quaternion group of order 8 and hence $(x,y)\in x^2 \GEN{a}$, as $D_2$ is normal in $G$. 
	
    Let $v$  be a generator of $A_2$ and $n_2=|A_2|$.
	For every $h\in G$ there is an odd integer $l_h$ such that $v^h=v^{l_h}$. 
	As $(h_{2'},v)=1$, and $h_2^2 \in \GEN{v,\nu}\subseteq D$ we have $l_h\equiv \pm 1 \mod \frac{n_2}{2}$. 
	Thus $(v^2)^h = v^{\pm 2}$ and if $l_h\equiv 1 \mod \frac{n_2}{2}$ then $(v^2)^{l_h} = v^2$.
	
	Therefore $x^2\equiv y^2 \equiv (x,y) = \nu \mod A$ and hence $1=\nu^2 = (x,y)^2 = (x^2,y)$ and $y^2\in \GEN{A,\nu}\subseteq D$.
	Thus $\nu\nu^y = (x,y)(x,y)^y = (x,y^2)=1$ and hence $(\nu,y)=1$. 
	As $\nu\not\in Z(G)$ and $\nu$ commutes with $x$, $y$ and the elements of odd order of $G$, there is a $2$-element $w$ of $G$ such that $\nu^w = v^{\frac{n_2}{2}}\nu$.
	Furthermore, there are integers $j,r$ and $s$ satisfying the following:
	$$x^y=x^{-1}v^j, \hspace{0.5cm} x^w = v^rx, \hspace{0.5cm} y^w = v^sy.$$
	Then 
	$$x=x^{y^2} = \left(x^{-1}v^j \right)^{-1} \left(v^j\right)^y = x v^{j(l_y-1)}.$$
	Hence 
	\begin{equation}\label{jl=j}
	n_2\mid j(l_y -1).
	\end{equation}
	Moreover
	$$v^{\frac{n_2}{2}} \nu = \nu^w = (x^w,y^w) = (v^rx,v^sy) = (v^r,y) \nu.$$
	Therefore $v^{\frac{n_2}{2}} = (v^r,y)$. In particular,
	\begin{equation}\label{lgNo1}
	l_y\not\equiv 1 \mod n_2
	\end{equation} and one of the following conditions holds:
	\begin{equation}\label{Posibleslgr}
	\matriz{{lllll}
		8\mid n_2, & l_y=1+\frac{n_2}{2} & \text{ and } & 2\nmid r & \text{ or } \\
		4\mid n_2, & l_y\equiv -1 \mod \frac{n_2}{2} & \text{ and } & r\equiv \pm \frac{n_2}{4} \mod n_2.} 
	\end{equation}
	From \eqref{jl=j} and \eqref{lgNo1} it follows that $j$ is even and hence $(v^j)^h=v^{\pm j}$ for every $h\in G$. 
	Moreover $\nu=(x,y)=x^{-2}v^j$ has order $2$ and hence $x^4 = v^{2j}$. 
	
	We claim that $l_y\equiv -1 \mod \frac{n_2}{2}$. Otherwise, by \eqref{Posibleslgr} we have $8\mid n_2$, $l_y=1+\frac{n_2}{2}$ and $r$ is odd. 
	Then 
	\begin{eqnarray*}
		x^{-1}v^{j+r(1+\frac{n_2}{2})} &=& 	(v^r x)^y= (v^r x)^{v^sy} = (x^w)^{y^w} = (x^y)^w = (x^{-1}v^j)^w \\ &=& (v^r x)^{-1} v^{\pm j} = x^{-1} v^{-r\pm j}.
	\end{eqnarray*}
	Thus $2r \equiv - j\pm j \mod \frac{n_2}{2}$. 
	As $r$ is odd and $4$ divides $\frac{n_2}{2}$, we have  $2\equiv 2r \equiv -j\pm j \equiv 0 \mod 4$, a contradiction.	
	Thus $j\equiv jl_g\equiv -j \mod n_2$, i.e. $n_2\mid 2j$. Then $x^4 = v^{2j}=1$ and $x^y \in x^{-1}\GEN{v^{\frac{n_2}{2}}}$. Therefore $x^2 \nu = x^2(x,y)=xx^y\in \GEN{v^{\frac{n_2}{2}}}$ and hence $x^2\in \nu^G$. 
	
	So $l_y\equiv -1 \mod \frac{n_2}{2}$, $4\mid n_2$ and $x^w = v^{\pm \frac{n_2}{4}} x$. 
	Moreover, if $8\nmid n_2$ then $l_y=-1$.
	Then $C_{\GEN{v}}(y)=\GEN{v^{\frac{n_2}{2}}}$. 
	Thus $y^2x^2\in C_{\GEN{v}}(y)=\GEN{v^{\frac{n_2}{2}}}$, i.e. $x^2=y^2 v^k$ with $k=0$ or $k=\frac{n_2}{2}$. 
	Then
	$$v^{\frac{n_2}{2}+k}y^2 = v^{\frac{n_2}{2}} x^2 = (v^{\frac{n_2}{4}}x)^2 = 
	(x^w)^2 = (y^w)^2 v^k = 
	(v^sy)^2 v^k = v^{s(1+l_y)+k}y^2$$
	Therefore $s(1+l_y)\equiv \frac{n_2}{2} \not\equiv 0 \mod n_2$.
	This is not compatible with $l_y=-1$ and hence $8\mid n_2$, $l_y\equiv \frac{n_2}{2}-1 \mod n_2$, i.e. $v^y=v^{\frac{n_2}{2}-1}$, and $s$ is odd, i.e. $A_2=\GEN{(y,w)}$.
	As $w^2\in A$ we have $x=x^{w^2} = v^{\pm (l_w+1)\frac{n_2}{4}}x$. Hence $l_w \equiv -1 \mod 4$ and hence $l_w\equiv -1 \mod \frac{n_2}{2}$.
	As $w$ commutes with $w^2$ we deduce that $w^2 \in \GEN{v^{\frac{n_2}{2}}}$ and hence $y=y^{w^2} = v^{(l_w+1)s} y$. As $s$ is odd and $8\mid n_2$ we have $l_w\equiv -1\mod n_2$, i.e. $v^w=v^{-1}$.

	Summarizing $8\mid n_2$ and $v,w,x,y$ satisfy the following properties:
	$$\matriz{{l}
		w^2 \in \GEN{v^{\frac{n_2}{2}}}, v^w = v^{-1}, |x|=4, (v,x)=1, x^w = v^{\pm \frac{n_2}{4}} x, x^2 y^2 \in \GEN{v^{\frac{n_2}{2}}}, \\
		v^y = y^{\frac{n_2}{2}-1}y, (w,y)=v^r, (x,y)\in x^2 \GEN{v^{\frac{n_2}{2}}}}$$
	Replacing $w$ by either $w^{-1}$, $x^2w$ or $x^2w^{-1}$ one may assume that $w^2=1$ and $x^w = v^{\frac{n_2}{4}}x$. Then replacing $y$ by $v^{\frac{n_2}{4}}y$ one may assume that $x^2=y^2$. 
	Finally, replacing $v$ by $v^r$ one may assume that $v=(w,y)$. This finishes the proof of Proposition~\ref{MinimalCounterexampleHamiltonian}.
\end{proofof}

\begin{proofof}{\it Theorem~\ref{CyclicByHamiltonian}}. 
Suppose that (ZC) fails for $G$ and let $N$ be a subgroup of $G$ which is maximal among the normal subgroups of $G$ such that (ZC) fails for $G/N$. Let us use bar notation for reduction modulo $N$. 
Then $\overline{G}$ and $\overline{A}$ satisfy the conditions of Proposition~\ref{MinimalCounterexampleHamiltonian} and therefore $|\overline{A}|$ is multiple of $8$, $[\overline{G}:\overline{A}]$ is multiple of $16$ and the image of the action by conjugation of $\overline{G}$ on the $2$-Sylow subgroup of $\overline{A}$ has order multiple of $4$. 
Then $|A|$ is multiple of $8$, $[G:A]$ is multiple of $16$ and the image of the action by conjugation of $G$ on $A_2$ has at least four elements. 
However, as $G/A$ is Hamiltonian, $(g^2,A_2)=1$ for every $g\in G$. This implies that the image of that action has at most four elements and hence it has exactly $4$ elements.
Thus $G$ satisfy neither of the conditions (1), (2) nor (3) of Theorem~\ref{CyclicByHamiltonian}. 
\end{proofof}

\textbf{Acknowledgment}. We thanks Wolfgang Kimmerle for suggesting us to study the Zassenhaus Conjecture for cyclic-by-Hamiltonian groups. 

\bibliographystyle{amsalpha}
\bibliography{ReferencesMSC}

\end{document}